\documentclass{amsproc}

\usepackage{amssymb}
\usepackage{amsfonts}
\usepackage[latin1]{inputenc}
\usepackage{color}

\newtheorem{theorem}{Theorem}[section]
\newtheorem{lemma}[theorem]{Lemma}
\newtheorem{proposition}[theorem]{Proposition}

\theoremstyle{definition}
\newtheorem{definition}[theorem]{Definition}
\newtheorem{example}[theorem]{Example}

\theoremstyle{remark}
\newtheorem{remark}[theorem]{Remark}

\numberwithin{equation}{section}

\begin{document}

\title{KMS STATES FOR THE GENERALIZED GAUGE ACTION ON GRAPH ALGEBRAS}

\author{Gilles G. de Castro \and Fernando de L. Mortari}
\address{Departamento de Matemática, Universidade Federal de Santa Catarina, 88040-970 Florianópolis SC, Brazil.}
\email{gilles.castro@ufsc.br \\ fernando.mortari@ufsc.br}
\thanks{Partially supported by Funpesquisa/UFSC\\
The first named author was also supported by project mathamsud U11-MATH05 and Capes/Math-AmSud 013/10}

\begin{abstract}
Given a positive function on the set of edges of an arbitrary directed graph $E=(E^0,E^1)$, we define a one-parameter group of automorphisms on the C*-algebra of the graph $C^*(E)$, and study the problem of finding KMS states for this action. We prove that there are bijective correspondences between KMS states on $C^*(E)$, a certain class of states on its core, and a certain class of tracial states on $C_0(E^0)$. We also find the ground states for this action and give some examples.
\end{abstract}

\maketitle

\section{Introduction}

Given a directed graph $E=(E^0,E^1)$, we can associate to it a C*-algebra $C^*(E)$, and an interesting problem that arises is to find relations between the algebraic properties of the algebra and the combinatorial properties of the graph \cite{MR2135030}. One such problem is to determine the set of KMS states for a certain action on the algebra.

Graph algebras are a generalization of Cuntz algebras and Cuntz-Krieger algebras. For the Cuntz algebra, there is a very natural action of the circle, the gauge action, which can be extended to an action of the real line. The KMS states for this action are studied in \cite{MR500150} and later generalized to a more general action of the line, that can be thought of as a generalized gauge action \cite{MR602475}. The same is done for the Cuntz-Krieger algebras \cite{MR759450}, \cite{MR2057042}.

Recently there were similar results proven for the C*-algebra associated to a finite graph. This is done in \cite{2010arXiv1007.4248K} for an arbitrary finite graph, in \cite{2010arXiv1002.0790I} for a certain class of finite graphs via groupoid C*-algebras and in \cite{2012arXiv1205.2194H} for the Toeplitz C*-algebra of the graph.

Our goal is to generalize these results to the case of an arbitrary graph. First we analyze which conditions the restrictions of a KMS state to the core of $C^*(E)$ and to $C_0(E^0)$ must satisfy. By using a description of the core as an inductive limit, we can build a KMS state on $C^*(E)$ from a tracial state on $C_0(E^0)$ satisfying the conditions found.

In section \ref{section_graph_algebra} we review some of the basic definitions and results about graph algebras as well as the description of the core as an inductive limit. In section \ref{section_KMS} we establish the results concerning KMS states, followed by a discussion on ground states in section \ref{section_ground_states}. In section \ref{section_examples}, some examples are given.

\section{Graph algebras}\label{section_graph_algebra}

\begin{definition}
A \emph{(directed) graph} $E=(E^0,E^1,r,s)$ consists of nonempty sets $E^0$, $E^1$ and functions $r,s:E^1\to E^0$; an element of $E^0$ is called a \emph{vertex} of the graph, and an element of $E^1$ is called an \emph{edge}. For an edge $e$, we say that $r(e)$ is the \emph{range} of $e$ and $s(e)$ is the \emph{source} of $e$.
\end{definition}

\begin{definition}A vertex $v$ in a graph $E$ is called a \emph{source} if $r^{-1}(v)=\emptyset$, and is said to be \emph{singular} if it is either a source, or $r^{-1}(v)$ is infinite.
\end{definition}

\begin{definition}
A \emph{path of length} $n$ in a graph $E$ is a sequence $\mu=\mu_1\mu_2\ldots\mu_n$ such that $r(\mu_i+1)=s(\mu_{i})$ for all $i=1,\ldots,n-1$. We write $|\mu|=n$ for the length of $\mu$ and regard vertices as paths of length $0$. We denote by $E^n$ the set of all paths of length $n$ and $E^*=\cup_{n\geq 0}E^n$. We extend the range and source maps to $E^*$ by defining $s(\mu)=s(\mu_n)$ and $r(\mu)=r(\mu_1)$ if $n\geq 2$ and $s(v)=v=r(v)$ for $n=0$.
\end{definition}

\begin{definition}
Given a graph $E$, we define the \emph{C*-algebra of} $E$ as the universal C*-algebra $C^*(E)$ generated by mutually orthogonal projections $\{p_v\}_{v\in E^0}$ and partial isometries $\{s_e\}_{e\in E^1}$ with mutually orthogonal ranges such that
\begin{enumerate}
\item $s_e^*s_e=p_{s(e)}$;
\item $s_es_e^*\leq p_r(e)$ for every $e\in E^1$;
\item $p_v=\sum_{e\in r^{-1}(v)}s_es_e^*$ for every $v\in E^0$ such that $0<|r^{-1}(v)|<\infty$.
\end{enumerate}

\end{definition}

For a path $\mu=\mu_1\ldots\mu_n$, we denote the composition $s_{\mu_1}\ldots s_{\mu_n}$ by $s_{\mu}$, and for $v\in E^0$ we define $s_v$ to be the projection $p_v$.

Propositions \ref{prop_multiplication}, \ref{prop_gauge_action}, \ref{prop_core_description} and \ref{prop_conditional_expectation} below are found in \cite{MR2135030} (as Corollary 1.15, Proposition 2.1, Proposition 3.2 and Corollary 3.3, respectively) in the context of row-finite graphs, but their proofs hold just the same for general graphs as above.

\begin{proposition}\label{prop_multiplication}
For $\alpha, \beta, \mu, \nu \in E^*$ we have
$$(s_{\mu} s_{\nu}^*)(s_{\alpha}s_{\beta}^*) = \left\{ 
\begin{array}{cc}
s_{\mu\alpha'}s_{\beta}^* & \mathrm{if}\ \alpha=\nu\alpha' \\
s_{\mu}s_{\beta\nu'}^* & \mathrm{if}\ \nu=\alpha\nu' \\
0 & \mathrm{otherwise}
\end{array}
\right. $$
and $C^*(E)=\overline{\mathrm{span}}\{s_\mu s_\nu^*:\mu,\nu\in E^*,\ s(\mu)=s(\nu)\}$.
\end{proposition}

\begin{proposition}\label{prop_gauge_action}
Let $E$ be a graph. Then there is an action $\gamma$ of $\mathbb{T}$ on $C^*(E)$, called a \emph{gauge action}, such that $\gamma_{z}(s_e)=zs_e$ for every $e\in E^1$ and $\gamma_{z}(p_v)=p_v$ for every $v\in E^0$.
\end{proposition}

\begin{definition}
The \emph{core} of the algebra $C^*(E)$ is the fixed-point subalgebra for the gauge action, denoted by $C^*(E)^{\gamma}$.
\end{definition}

\begin{proposition}\label{prop_core_description}
$C^*(E)^{\gamma}=\overline{\mathrm{span}}\{s_\mu s_\nu^*:\mu,\nu\in E^*,\ s(\mu)=s(\nu),\ |\mu|=|\nu|\}$.
\end{proposition}

\begin{proposition}\label{prop_conditional_expectation} There is a conditional expectation $\Phi:C^*(E)\to C^*(E)^{\gamma}$ such that $\Phi(s_{\mu}s_{\nu}^*)=[|\mu|=|\nu|]s_{\mu}s_{\nu}^*$.
\end{proposition}

It is useful to describe the core as an inductive limit of subalgebras, as was done in an appendix in \cite{2011arXiv1107.1019H}. The idea is as follows. For $k\geq 0$ define the sets

$$F_k=\overline{\mathrm{span}}\{s_\mu s_\nu^*:\mu,\nu\in E^k,\ s(\mu)=s(\nu)\},$$
$$\mathcal{E}_k=\overline{\mathrm{span}}\{s_\mu s_\nu^*:\mu,\nu\in E^k\ {\rm and}\ s(\mu)=s(\nu)\ {\rm is\ singular}\},$$
$$C_k=F_0+\cdots+F_k.$$

Also, for a given vertex $v$ we define
$$F_k(v)=\overline{\mathrm{span}}\{s_\mu s_\nu^*:\mu,\nu\in E^k,\ s(\mu)=s(\nu)=v\}$$
so that
\begin{equation}\label{eq:decompostionFk}
F_k=\bigoplus_{v\in E^0}F_k(v)
\end{equation}
as a direct sum of C*-algebras.

\begin{lemma}\label{lemma_approxunit}
Let $\Lambda$ be the set of all finite subsets of $E^k$ and for $\lambda\in\Lambda$ define
$$u_{\lambda}=\sum_{\mu\in\lambda}s_{\mu}s_{\mu}^*.$$
Then $\{u_{\lambda}\}_{\lambda\in\Lambda}$ is an approximate unit of $F_k$ consisting of projections.
\end{lemma}

\begin{proof}
This is a direct consequence of Proposition \ref{prop_multiplication}.
\end{proof}

The following result is a combination of Proposition A.1 and Lemma A.2 in \cite{2011arXiv1107.1019H}.

\begin{proposition}\label{prop_Ck_description}With the notation as above for a graph $E$, the following hold for $k\geq 0$:
\begin{itemize}
\item[(a)] $C_k$ is a C*-subalgebra of $C^*(E)^{\gamma}$, $F_{k+1}$ is an ideal in $C_k$, $C_k\subseteq C_{k+1}$ and 
$$C^*(E)^{\gamma}=\varinjlim C_k.$$
\item[(b)] $F_k\cap F_{k+1}=\bigoplus\{F_k(v): 0<|r^{-1}(v)|<\infty\}.$ (C*-algebraic direct sum)
\item[(c)] $C_k=\mathcal{E}_0\oplus\mathcal{E}_1\oplus\cdots\oplus\mathcal{E}_{k-1}\oplus F_k$ (vector space direct sum)
\end{itemize}
\end{proposition}

With the above, we can now prove the following.

\begin{proposition}\label{prop_intersection} For each $k\geq 0$, $C_k\cap F_{k+1}=F_k\cap F_{k+1}$.
\end{proposition}
\begin{proof}
Obviously one has $F_k\cap F_{k+1}\subseteq C_k\cap F_{k+1}$. On the other hand, given $x\in C_k\cap F_{k+1}$, one can decompose $x$ as sums in $C_k$ and $C_{k+1}$ with Proposition \ref{prop_Ck_description}(c), use the fact that $F_k=\mathcal{E}_k\oplus F_k\cap F_{k+1}$ and the uniqueness of the direct sum decompositions of $x$ to conclude that $x\in F_k$.
\end{proof}

\section{KMS states for the generalized gauge action}\label{section_KMS}

In this section we define an action on $C^*(E)$ from a function $c:E^1\to \mathbb{R}^*_+$ similar to what is done in \cite{MR602475} for the Cuntz algebras and in \cite{MR2057042} for the Cuntz-Krieger algebras. We will always suppose that there is a constant $k>0$ such that $c(e)>k$ for all $e\in E^1$ and that $\beta>0$. Observe that in this case $c^{-\beta}$ is bounded.

We extend a function as above to a function $c:E^*\to\mathbb{R}^*_+$ by defining $c(v)=1$ if $v\in E^0$ and $c(\mu)=c(\mu_1)\ldots c(\mu_n)$ if $\mu=\mu_1\cdots\mu_n\in E^n$.

\begin{proposition}
Given a function $c:E^1\to\mathbb{R}^*_+$, there is a strongly continuous action $\sigma^c:\mathbb{R}\to\mathrm{Aut}(C^*(E))$ given by $\sigma_t^c(p_v)=p_v$ for all $v\in E^0$ and $\sigma_t^c(s_e)=c(e)^{it}s_e$ for all $e\in E^1$.
\end{proposition}

\begin{proof}
Let $T_e=c(e)^{it}s_e$ and note that $T_e$ is a partial isometry with $T_e^*T_e=s_e^*s_e$ and $T_eT_e^*=s_es_e^*$. It follows that the sets $\{p_v\}_{v\in E^0}$ and $\{T_e\}_{e\in E^1}$ satisfy the same relations as $\{p_v\}_{v\in E^0}$ and $\{s_e\}_{e\in E^1}$. By the universal property, there is a homomorphism $\sigma_t^c:C^*(E)\to C^*(E)$ such that $\sigma_t^c(p_v)=p_v$ for all $v\in E^0$ and $\sigma_t^c(s_e)=T_e=c(e)^{it}s_e$ for all $e\in E^1$.

It is easy to see that $\sigma_{t_1}^c\circ\sigma_{t_2}^c=\sigma_{t_1+t_2}^c$ and $\sigma_0^c=Id$. Hence $\sigma_t^c$ is an automorphism with inverse $\sigma_{-t}^c$.

To prove continuity, let $a\in C^*(E)$, $t\in\mathbb{R}$ and $\varepsilon>0$. Take $x$ to be a finite sum $x=\sum_{\mu,\nu\in E^*}\lambda_{\mu,\nu}s_{\mu}s_{\nu}^*$ such that $\|a-x\|<\varepsilon/3$. For each pair of paths $\mu,\nu$ with $\lambda_{\mu,\nu}\neq 0$, there is $\delta_{\mu,\nu}$ such that 

$$|c(\mu)^{it}c(\nu)^{-it}-c(\mu)^{iu}c(\nu)^{-iu}|<\frac{\varepsilon}{3\sum_{\mu,\nu\in E^*}\|\lambda_{\mu,\nu}s_{\mu}s_{\nu}^*\|}$$
for all $u\in\mathbb{R}$ with $|t-u|<\delta_{\mu,\nu}$. If we take $\delta$ to be the minimum of all such $\delta_{\mu,\nu}$, then for all $u\in\mathbb{R}$ with $|t-u|<\delta$ we have

$$\|\sigma_t^c(x)-\sigma_u^c(x)\|=\left\|\sum_{\mu,\nu\in E^*}(c(\mu)^{it}c(\nu)^{-it}-c(\mu)^{iu}c(\nu)^{-iu})\lambda_{\mu,\nu}s_{\mu}s_{\nu}^*
\right\|<$$
$$\frac{\varepsilon}{3\sum_{\mu,\nu\in E^*}\|\lambda_{\mu,\nu}s_{\mu}s_{\nu}^*\|}\sum_{\mu,\nu\in E^*}\|\lambda_{\mu,\nu}s_{\mu}s_{\nu}^*\|=\frac{\varepsilon}{3}$$
and hence

$$\|\sigma_t^c(a)-\sigma_u^c(a)\|=\|\sigma_t^c(a)-\sigma_t^c(x)+\sigma_t^c(x)-\sigma_u^c(x)+\sigma_u^c(x)-\sigma_u^c(a)\|\leq$$
$$\leq\|\sigma_t^c(a-x)\|+\|\sigma_t^c(x)-\sigma_u^c(x)\|+\|\sigma_u^c(x-a)\|\leq \varepsilon/3+\varepsilon/3+\varepsilon/3=\varepsilon.$$

\end{proof}

From now on, we will write simply $\sigma$ instead of $\sigma_c$. The next result shows that KMS states on $C^*(E)$ are determined by their values at the core algebra.

\begin{proposition}\label{prop_unicity_kms_core}
Suppose $c:E^1\to\mathbb{R}^*_+$ is such that $c(\mu)\neq 1$ for all $\mu\in E^*\backslash E^0$. If two $(\sigma,\beta)$-KMS states $\varphi_1,\varphi_2$ on $C^*(E)$ coincide at the core algebra $C^*(E)^{\gamma}$, then $\varphi_1=\varphi_2$.
\end{proposition}

\begin{proof}
Taking an arbitrary $s_{\mu} s_{\nu}^*$ such that $s(\mu)=s(\nu)$, if $|\mu|=|\nu|$ then $s_{\mu} s_{\nu}^*\in C^*(E)^{\gamma}$ and thus $\varphi_1(s_{\mu} s_{\nu}^*)=\varphi_2(s_{\mu} s_{\nu}^*)$.

Suppose then that $|\mu|\neq |\nu|$, and denote the functional $\varphi_2-\varphi_1$ by $\varphi$. Using the KMS condition, one obtains
$$\varphi(s_{\mu} s_{\nu}^*)=\varphi(s_{\nu}^*c(\mu)^{-\beta}s_{\mu}) = \left\{ 
\begin{array}{cc}
c(\mu)^{-\beta}\varphi(s_{\nu'}^*) & \mathrm{if}\ \nu=\mu\nu'\\
c(\mu)^{-\beta}\varphi(s_{\mu'}^*) & \mathrm{if}\ \mu=\nu\mu'\\
0 & \mathrm{otherwise}
\end{array}
\right. .$$
It is therefore sufficient to show that $\varphi(s_{\mu})=\varphi(s_{\mu}^*)=0$ if $|\mu|\geq 1$. To see this, notice that if $C^*(E)$ has a unit, then $$\varphi(s_{\mu})=\varphi(s_{\mu}1)=\varphi(1c(\mu)^{-\beta}s_{\mu})=c(\mu)^{-\beta}\varphi(s_{\mu}),$$ whence $\varphi(s_{\mu})=0$ since $c(\mu)\neq 1$ by hypothesis; the non-unital case is established analogously with the use of an approximate unit.
\end{proof}

\begin{theorem}\label{thm:KMSparte1}
Suppose $c:E^1\to\mathbb{R}^*_+$ is such that $c(\mu)\neq 1$ for all $\mu\in E^*\backslash E^0$. If $\varphi$ is a $(\sigma,\beta)$-KMS state on $C^*(E)$ then its restriction $\omega=\varphi|_{C^*(E)^{\gamma}}$ to $C^*(E)^{\gamma}$ satisfies
\begin{equation}\label{eq:CoreCondition}
\omega(s_{\mu}s_{\nu}^*)=[\mu=\nu]c(\mu)^{-\beta}\omega(p_{s(\mu)});
\end{equation}
conversely, if $\omega$ is a state on $C^*(E)^{\gamma}$ satisfying \emph{(\ref{eq:CoreCondition})} then $\varphi=\omega\circ\Phi$ is a $(\sigma,\beta)$-KMS state on $C^*(E)$, where $\Phi$ is the conditional expectation from proposition \ref{prop_conditional_expectation}. The correspondence thus obtained is bijective and preserves convex combinations.
\end{theorem}

\begin{proof}
Let $\varphi$ be a $(\sigma,\beta)$-KMS state on $C^*(E)$ and $\omega$ its restriction to $C^*(E)^{\gamma}$. If $\mu,\nu$ are paths such that $|\mu|=|\nu|$ and $s(\mu)=s(\nu)$ then
$$\omega(s_{\mu}s_{\nu}^*)=\varphi(s_{\mu}s_{\nu}^*)=\varphi(s_{\nu}^*\sigma_{i\beta}(s_{\mu}))=\varphi(s_{\nu}^*c(\mu)^{-\beta}s_{\mu})=$$
$$=[\mu=\nu]c(\mu)^{-\beta}\varphi(p_{s(\mu)})=[\mu=\nu]c(\mu)^{-\beta}\omega(p_{s(\mu)}).$$

Conversely, let $\omega$ be a state on $C^*(E)^{\gamma}$ satisfying (\ref{eq:CoreCondition}) and $\varphi=\omega\circ\Phi$; we have to show that $\varphi$ satisfies the KMS condition. By continuity and linearity, it is sufficient to verify this for elements $x=s_{\mu}s_{\nu}^*$ and $y=s_{\zeta}s_{\eta}^*$ where $\mu,\nu,\zeta,\eta\in E^*$ are paths such that $s(\mu)=s(\nu)$ and $s(\zeta)=s(\eta)$.

We need to check that $\varphi(xy)=\varphi(y\sigma_{i\beta}(x))$. First note that

$$xy=(s_{\mu} s_{\nu}^*)(s_{\zeta}s_{\eta}^*) = \left\{ 
\begin{array}{ccc}
s_{\mu\zeta'}s_{\eta}^* & \mathrm{if}\ \zeta=\nu\zeta' & (1) \\
s_{\mu}s_{\eta\nu'}^* & \mathrm{if}\ \nu=\zeta\nu' & (2) \\
0 & \mathrm{otherwise} & (3)
\end{array}
\right. $$
and
$$y\sigma_{i\beta}(x)=c(\mu)^{-\beta}c(\nu)^{\beta}(s_{\zeta} s_{\eta}^*)(s_{\mu}s_{\nu}^*) = c(\mu)^{-\beta}c(\nu)^{\beta}\left\{ 
\begin{array}{ccc}
s_{\zeta\mu'}s_{\nu}^* & \mathrm{if}\ \mu=\eta\mu' & (a)\\
s_{\zeta}s_{\nu\eta'}^* & \mathrm{if}\ \eta=\mu\eta' & (b)\\
0 & \mathrm{otherwise} & (c)
\end{array}
\right. .$$

There are nine cases to consider. In each case it must be checked whether the resulting paths have the same size, for they will be otherwise sent to $0$ by $\Phi$.

Case 1-a. In this case $\zeta=\nu\zeta'$ and $\mu=\eta\mu'$ so that $|\zeta|=|\nu|+|\zeta'|$ and $|\mu|=|\eta|+|\mu'|$. We claim that $|\mu\zeta'|=|\mu|+|\zeta'|=|\eta|$ if and only if $|\zeta\mu'|=|\zeta|+|\mu'|=|\nu|$, and in this case $\mu=\eta$ and $\nu=\zeta$. In fact,
$$|\mu|+|\zeta'|=|\eta| \Leftrightarrow |\eta|+|\mu'|+|\zeta'|=|\eta| \Leftrightarrow |\mu'|+|\zeta'|=0 \Leftrightarrow$$
$$\Leftrightarrow |\nu|+|\zeta'|+|\mu'|=|\nu| \Leftrightarrow |\zeta|+|\mu'|=|\nu|.$$
Observe that, in this case, we have $|\mu'|+|\zeta'|=0$ so that $|\mu'|=|\zeta'|=0$, and hence $\mu=\eta$, $\nu=\zeta$.

It follows that, if $|\mu\zeta'|\neq|\eta|$, then $$\varphi(xy)=\omega\circ\Phi(xy)=\omega(0)=\omega\circ\Phi(y\sigma_{i\beta}(x))=\varphi(y\sigma_{i\beta}(x))$$
and, if $|\mu\zeta'|=|\eta|$, we get
$$\varphi(xy)=\varphi(s_{\mu}s_{\mu}^*)=\omega(s_{\mu}s_{\mu}^*)=c(\mu)^{-\beta}\omega(p_{s(\mu)})$$
and on the other hand
$$\varphi(y\sigma_{i\beta}(x))=c(\mu)^{-\beta}c(\nu)^{\beta}\varphi(s_{\nu}s_{\nu}^*)=c(\mu)^{-\beta}c(\nu)^{\beta}\omega(s_{\nu}s_{\nu}^*)=$$
$$=c(\mu)^{-\beta}c(\nu)^{\beta}c(\nu)^{-\beta}\omega(p_{s(\mu)})=c(\mu)^{-\beta}\omega(p_{s(\mu)}).$$

Case 1-b. Now, we have that $\zeta=\nu\zeta'$ and $\eta=\mu\eta'$ so that $|\zeta|=|\nu\zeta'|=|\nu|+|\zeta'|$ and $|\eta|=|\mu\eta'|=|\mu|+|\eta'|$; as before, we can check that $|\mu|+|\zeta'|=|\eta|$ if and only if $|\zeta|=|\nu|+|\eta'|$. If that is not the case then $\varphi(xy)=0=\varphi(y\sigma_{i\beta}(x))$. If the equivalent conditions are true then
$$\varphi(xy)=\varphi(s_{\mu\zeta'}s_{\eta}^*)=\omega(s_{\mu\zeta'}s_{\eta}^*)=[\mu\zeta'=\eta]c(\eta)^{-\beta}\omega(p_{s(\eta)})$$
and
$$\varphi(y\sigma_{i\beta}(x))=c(\mu)^{-\beta}c(\nu)^{\beta}\varphi(s_{\zeta}s_{\nu\eta'}^*)=c(\mu)^{-\beta}c(\nu)^{\beta}[\zeta=\nu\eta']c(\zeta)^{-\beta}\omega(p_{s(\zeta)}).$$
Since $\zeta=\nu\zeta'$ and $\eta=\mu\eta'$, we have that $\mu\zeta'=\eta$ if and only if $\zeta=\nu\eta'$ and if both are true, then $\zeta'=\eta'$ and 
$$c(\mu)^{\beta}c(\nu)^{-\beta}c(\zeta)^{-\beta}=c(\mu)^{-\beta}c(\nu)^{\beta}c(\nu)^{-\beta}c(\eta')^{-\beta}=c(\mu)^{-\beta}c(\eta')^{-\beta}=$$
$$=c(\mu)^{-\beta}c(\zeta')^{-\beta}=c(\eta)^{-\beta}.$$
From our original hypothesis, we have that $s(\eta)=s(\zeta)$ so we conclude that $\varphi(xy)=\varphi(y\sigma_{i\beta}(x))$.

Case 1-c. In this case $\varphi(y\sigma_{i\beta}(x))=0$, so we need to check that $\varphi(xy)=0$. As with the previous case, we have that $\varphi(xy)=[\mu\zeta'=\eta]c(\eta)^{-\beta}\omega(p_{s(\eta)})$; however, in case (c) $\mu\zeta'\neq\eta$ for all $\zeta'$ and therefore $\varphi(xy)=0$.

The other cases are analogous to these three, except for case 3-c, where $\varphi(xy)=0=\varphi(y\sigma_{i\beta}(x))$ since $xy=0=y\sigma_{i\beta}(x)$.

That the correspondence obtained is bijective follows from Proposition \ref{prop_unicity_kms_core} and that it preserves convex combinations is immediate.
\end{proof}

Next, we want to show that there is also a bijective correspondence between $(\sigma,\beta)$-KMS states on $C^*(E)$ and a certain class of tracial states on $C_0(E^0)$. We build this correspondence by first describing a correspondence between this class of tracial states on $C_0(E^0)$ and states $\omega$ on $C^*(E)^{\gamma}$ satisfying (\ref{eq:CoreCondition}).

The conditions found for the states on $C_0(E^0)$ are similar to those in \cite{MR2056837}, although as discussed in \cite{2010arXiv1007.4248K}, their results cannot be used directly for an arbitrary graph; nevertheless, the results of Theorem 1.1 of \cite{MR2056837} still apply in the general setting, and we use them to build a certain kind of transfer operator on the dual of $C_0(E^0)$.

Let us first recall how to construct $C^*(E)$ as C*-algebra associated to a C*-correspondence \cite{MR2029622}. If we let $A=C_0(E^0)$, then $C_c(E^1)$ has a pre-Hilbert A-module structure given by
$$\left<\xi,\eta\right>(v)=\sum_{e\in s^{-1}(v)}\overline{\xi(e)}\eta(e)\quad\mathrm{for}\quad v\in E^0,$$
$$(\xi a)(e)=\xi(e)a(s(e))\quad\mathrm{for}\quad e\in E^1,$$
where $\xi,\eta\in C_c(E^1)$ and $a\in A$; it follows that the completion $X$ of $C_c(E^1)$ with respect to the norm given by $\|\xi\|=\|\left<\xi,\xi\right>\|^{1/2}$ is a Hilbert A-module. A representation $i_X:A\to\mathcal{L}(X)$ is then defined by by
$$i_X(a)(\xi)(e)=a(r(e))\xi(e)\quad\mathrm{for}\quad v\in E^0,$$
where $\mathcal{L}(X)$ is the C*-algebra of adjointable operators on $X$.

Let $\mathcal{K}(X)$ be the C*-subalgebra of $\mathcal{L}(X)$ generated by the operators $\theta_{\xi,\eta}$ given by $\theta_{\xi,\eta}(\zeta)=\xi\left<\eta,\zeta\right>$. For each $e\in E^1$, let $\chi_e\in C_c(E^1)$ be the characteristic function of $\{e\}$ and observe that 
$$\left\{t_{\lambda}=\sum_{e\in\lambda}\theta_{\chi_e,\chi_e}\right\}_{\lambda\in\Lambda},$$
where $\Lambda$ is the set of all finite subsets of $E^1$, is an approximate unit of $\mathcal{K}(X)$ . It is essentially the same approximate unit given by Lemma \ref{lemma_approxunit}.

If $\tau$ is a tracial state on $C_0(E^0)$, as in Theorem 1.1 of \cite{MR2056837} we define a trace $\mathrm{Tr}_{\tau}$ on $\mathcal{L}(X)$ by
$$\mathrm{Tr}_{\tau}(T)=\lim_{\lambda\to\infty}\sum_{e\in\lambda}\tau\left(\left<\chi_e,T\chi_e\right>\right)$$
where $T\in\mathcal{L}(X)$.

For a function $c:E^1\to\mathbb{R}^*_+$ as in the beginning of the section and $\beta>0$, we have that $c^{-\beta}\in C_b(E^1)$ and so it defines an operator on $\mathcal{L}(X)$ by pointwise multiplication.

\begin{definition}
Given $c$ and $\beta$ as above and $\tau$ a tracial state on $C_0(E^0)$, we define a trace $\mathcal{F}_{c,\beta}(\tau)$ on $C_0(E^0)$ by
$$\mathcal{F}_{c,\beta}(\tau)(a)=\mathrm{Tr}_{\tau}(i_X(a)c^{-\beta}).$$
\end{definition}

Now, observe that $C_0(E^0)\cong\overline{\mathrm{span}}\{p_v\}_{v\in E^0}$; regarding this as an equality, for a given tracial state $\tau$ on $C_0(E^0)$ we will write $\tau(p_v)=\tau_v$. For $v\in E^0$, it can be verified that
$$\mathcal{F}_{c,\beta}(\tau)(p_v)=\lim_{D\to r^{-1}(v)}\sum_{e\in D} c(e)^{-\beta}\tau_{s(e)},$$
where the limit is taken on finite subsets $D$ of $r^{-1}(v)$, and $\mathcal{F}_{c,\beta}(\tau)(p_v)=0$ if $r^{-1}(v)=\emptyset$.

\begin{remark}\label{remark_F}
By Theorem 1.1 of \cite{MR2056837}, if $\mathcal{F}_{c,\beta}(\tau)(a)<\infty$ for all $a\in C_0(E^0)$, then $\mathcal{F}_{c,\beta}(\tau)$ is actually a positive linear functional; also, if $V\subseteq E^0$ and $\mathcal{F}_{c,\beta}(\tau)(p_v)<\infty$ for all $v\in V$ then $\mathcal{F}_{c,\beta}(\tau)$ is a positive linear functional on $\overline{\mathrm{span}}\{p_v:v\in V\}$.
\end{remark}

\begin{definition}
For a vertex $v\in E^0$ and a positive integer $n$, we define
$$r^{- n}(v)=\{\mu\in E^{n}:r(\mu)=v\}.$$
\end{definition}

\begin{lemma}\label{lemma_inequalityF}
If $\mathcal{F}_{c,\beta}(\tau)(p_v)\leq \tau_v$ for all $v\in E^0$ then
$$\lim_{D\to r^{-n}(v)}\sum_{\mu\in D} c(\mu)^{-\beta}\tau_{s(\mu)}\leq \tau_v$$
for all $v\in E^0$ and for all $n\in\mathbb{N}^*$.
\end{lemma}

\begin{proof}
This is proved by induction. The case $n=1$ is the hypothesis. Now suppose it is true for $n$, then
$$\lim_{D\to r^{-(n+1)}(v)}\sum_{\mu\in D} c(\mu)^{-\beta}\tau_{s(\mu)}=\lim_{D\to r^{-n}(v)}\sum_{\nu\in r^{\leq n}}c(\nu)^{-\beta}\sum_{e\in r^{-1}(s(\nu))}c(e)^{-\beta}\tau_{s(e)} [\nu e\in D]\leq$$
$$\leq \lim_{D\to r^{-n}(v)}\sum_{\nu\in D} c(\nu)^{-\beta}\tau_{s(\nu)}\leq \tau_v$$
where the first inequality is true due to the fact that since $c$ is a positive function then the net $\sum_{e\in D}c(e)^{-\beta}\tau_{s(e)}$ for finite subsets $D$ of $r^{-1}(s(\nu))$ is nondecreasing and less than or equal to $\tau_{s(\nu)}$ by hypothesis. The last inequality is the induction hypothesis.
\end{proof}

The next lemma is found in \cite{MR1953065} for unital algebras, but their proof carries out the same in the non-unital case by using an approximate unit instead of a unit.

\begin{lemma}[Exel-Laca]\label{lemma_exel_laca}
Let $B$ be a C*-algebra, $A$ be a C*-subalgebra such that an approximate unit of $A$ is also an approximate unit of $B$ and $I$ a closed bilateral ideal of $B$ such that $B=A+I$. Let $\varphi$ be a state on $A$ and $\psi$ a linear positive functional on $I$ such that $\varphi(x)=\psi(x)$ $\forall x\in A\cap I$ and $\overline{\psi}(x)\leq\varphi(x)$ $\forall x\in A^+$, where $\overline{\psi}(x)=\lim_{\lambda}\psi(bu_{\lambda})$ for an approximate unit $\{u_\lambda\}_{\lambda\in\Lambda}$ of $I$. Then there is a unique state $\Phi$ on $B$ such that $\Phi|_A=\varphi$ and $\Phi|_I=\psi$.
\end{lemma}

We want to use this lemma for $A=C_n$, $I=F_{n+1}$ and $B=C_{n+1}$, defined in section \ref{section_graph_algebra}. For that, we first note that $F_{n+1}$ is indeed an ideal of $C_{n+1}$ by Proposition \ref{prop_Ck_description} and that the approximate unit for $F_0$ given by Lemma \ref{lemma_approxunit} is also an approximate unit of $C_n$ for all $n$. We also need to know what the intersection $A\cap I$ is, and for that we need a preliminary result.

\begin{lemma}\label{lemma_funcionalFk}
Suppose $c$, $\beta$ and $\tau$ are such that $\mathcal{F}_{c,\beta}(\tau)(a)\leq\tau(a)$ for all $a\in C_0(E^0)^+$, then for each $k\geq 1$ there is a unique positive linear functional $\psi_k$ on $F_k$ defined by
\begin{equation}\label{eq:functionalFk}
\psi_k(s_{\mu}s_{\nu}^*)=[\mu=\nu]c(\mu)^{-\beta}\tau_{s(\mu)}.
\end{equation}
\end{lemma}

\begin{proof}
Since $\{s_\mu s_\nu^*:\mu,\nu\in E^k,\ s(\mu)=s(\nu)\}$ is linearly independent, equation \ref{eq:functionalFk} defines a unique linear functional on $\mathrm{span}\{s_\mu s_\nu^*:\mu,\nu\in E^k,\ s(\mu)=s(\nu)\}$. To extend to the closure, it is sufficient to prove that $\psi_k$ is continuous.

If $x\in\mathrm{span}\{s_\mu s_\nu^*:\mu,\nu\in E^k,\ s(\mu)=s(\nu)\}$ then
$$x=\sum_{v\in V}\sum_{(\mu,\nu)\in G_v}a_{\mu,\nu}^v s_{\mu}s_{\nu}^*$$
where $V$ is a finite subset of $E^0$ and $G_v$ is a finite subset of $\{(\mu,\nu)\in E^n\times E^n:s(\mu)=s(\nu)=v \}$. Using the decomposition given by equation \ref{eq:decompostionFk} and observing that $\{s_{\mu}s_{\nu}^*:(\mu,\nu)\in G_v\}$ can be completed to generators of a matrix algebra, we have that
$$\|x\|=\max_{v\in V}\left\|\sum_{(\mu,\nu)\in G_v}a_{\mu,\nu}^v s_{\mu}s_{\nu}^*\right\|=\max_{v\in V}\|(a_{\mu,\nu}^v)_{\mu,\nu}\|$$
where the last norm is the matrix norm.

If $\mathrm{Tr}$ is the usual matrix trace we have
$$|\psi_k(x)|=\left|\psi_k\left(\sum_{v\in V}\sum_{(\mu,\nu)\in G_v}a_{\mu,\nu}^v s_{\mu}s_{\nu}^* \right)\right|=$$
$$=\left|\sum_{v\in V}\sum_{(\mu,\nu)\in G_v}a_{\mu,\nu}^v [\mu=\nu]c(\mu)^{-\beta}\tau_{s(\mu)}\right|=$$
$$=\left|\sum_{v\in V}\mathrm{Tr}((a_{\mu,\nu}^v)_{\mu,\nu}\mathrm{diag}(c(\mu)^{-\beta}\tau_{s(\mu)})\right|\leq$$
$$\leq\sum_{v\in V}\left|\mathrm{Tr}((a_{\mu,\nu}^v)_{\mu,\nu}\mathrm{diag}(c(\mu)^{-\beta}\tau_{s(\mu)})\right|\leq$$
$$\leq\sum_{v\in V}\|(a_{\mu,\nu}^v)_{\mu,\nu}\|\sum_{\mu:(\mu,\mu)\in G_v}c(\mu)^{-\beta}\tau_{s(\mu)}\leq^{\mathrm{lemma\ \ref{lemma_inequalityF}}}$$
$$\leq\sum_{v\in V}\|(a_{\mu,\nu}^v)_{\mu,\nu}\|\tau_v\leq\max_{v\in V}(\|(a_{\mu,\nu}^v)_{\mu,\nu}\|)\sum_{v \in V}\tau _v=$$
$$= \|x\|\sum_{v \in V}\tau _v\leq \|x\|$$
where the last inequality comes from the fact that $\tau$ comes from a probability measure on a discrete space.
\end{proof}

\begin{theorem}\label{thm:KMSparte2}
If $\omega$ is a state on $C^*(E)^{\gamma}$ satisfying (\ref{eq:CoreCondition}) then its restriction $\tau$ to $C_0(E^0)$ satisfies:
\begin{itemize}
\item[(K1)] $\mathcal{F}_{c,\beta}(\tau)(a)=\tau(a)$ for all $a\in\overline{\mathrm{span}}\{p_v:0<|r^{-1}(v)|<\infty\}$,
\item[(K2)] $\mathcal{F}_{c,\beta}(\tau)(a)\leq\tau(a)$ for all $a\in C_0(E^0)^+$.
\end{itemize}
Conversely, if $\tau$ is a tracial state on $C_0(E^0)$ satisfying (K1) and (K2) then there is unique state $\omega$ on $C^*(E)^{\gamma}$ satisfying (\ref{eq:CoreCondition}). This correspondence preserves convex combinations.
\end{theorem}

\begin{proof}
Let $\omega$ be a state on $C^*(E)^{\gamma}$ satisfying (\ref{eq:CoreCondition}) and $\tau$ its restriction to $C_0(E^0)$. By Remark \ref{remark_F}, to establish (K1) it is sufficient to consider $a=p_v$ where $v\in E^0$ is such that $0<|r^{-1}(v)|<\infty$, and in this case

$$\tau(p_v)=\omega(p_v)=\omega\left(\sum_{e\in r^{-1}(v)}s_es_e^*\right)=\sum_{e\in r^{-1}(v)}c(e)^{-\beta}\omega(p_{s(e)})=$$
$$=\sum_{e\in r^{-1}(v)}c(e)^{-\beta}\tau_{s(e)}=\mathcal{F}_{c,\beta}(\tau)(p_v).$$

For (K2), let $a\in C_0(E^0)^+$ and write $a=\sum_{v\in E^0}{a_vp_v}$; again, by remark \ref{remark_F} it is sufficient to show the result for $a=p_v$ where $v\in E^0$. If $0<|r^{-1}(v)|<\infty$, then we have an equality as shown above. If $|r^{-1}(v)|=0$, then $\mathcal{F}_{c,\beta}(\tau)(p_v)=0\leq \tau(p_v)$. If $|r^{-1}(v)|=\infty$, then

$$\mathcal{F}_{c,\beta}(\tau)(p_v)=\lim_{D\to r^{-1}(v)}\sum_{e\in D} c(e)^{-\beta}\tau_{s(e)}=\lim_{D\to r^{-1}(v)}\sum_{e\in D}\omega(s_es_e^*)=$$
$$=\lim_{D\to r^{-1}(v)}\sum_{e\in D}\omega(p_vs_es_e^*)\leq \omega(p_v)=\tau(p_v).$$

To see the inequality above, we observe that $s_es_e^*$ are mutually orthogonal projections that commute with $p_v$ so that
$$p_v-\sum_{e\in D}p_vs_es_e^*=p_v\left(1-\sum_{e\in D}s_es_e^*\right)=\left(1-\sum_{e\in D}s_es_e^*\right)p_v\left(1-\sum_{e\in D}s_es_e^*\right)\geq 0.$$

Now, let $\tau$ be a tracial state on $C_0(E^0)$ satisfying (K1) and (K2). We will use Lemma \ref{lemma_exel_laca} and the discussion after it. Observe that $F_0=C_0(E^0)$ and let $\psi_0=\tau$. For $n\geq 1$, by Lemma \ref{lemma_funcionalFk} there exists a positive linear functional $\psi_n$ on $F_n$ defined by
$$\psi_n(s_{\mu}s_{\nu}^*)=[\mu=\nu]c(\mu)^{-\beta}\tau_{s(\mu)}.$$

Let us show by induction that there is a unique state $\varphi_n$ on $C_n$ such that the restriction to $F_n$ is $\psi_n$. For $n=1$, we use Lemma \ref{lemma_exel_laca} with $A=C_0(E^0)$, $I=F_1$, $B=C_1$, $\varphi=\tau$ and $\psi=\psi_1$. By Proposition \ref{prop_intersection}, in this case $A\cap I=\overline{\mathrm{span}}\{p_v:v\in E^0,\ 0<|r^{-1}(v)|<\infty\}$ and if $p_v\in A\cap I$ then
$$\psi(p_v)=\psi_1(p_v)=\psi_1(s_vs_v^*)=\tau_v=\tau(p_v).$$
Using the approximate unit given by Lemma \ref{lemma_approxunit}, for any $v\in E^0$ we have
$$\overline{\psi}(p_v)=\overline{\psi_1}(p_v)=\lim_{\lambda\to\infty}\psi_1(p_vu_{\lambda})=\lim_{D\to r^{-1}(v)}\sum_{e\in D}\psi_1(s_es_e^*)=$$
$$=\lim_{D\to r^{-1}(v)}\sum_{e\in D}c(e)^{-\beta}\tau_{s(e)}=\mathcal{F}_{c,\beta}(\tau)(p_v)\leq\tau(p_v),$$
where the last inequality is exactly (K2).

Now suppose that there is a unique state $\varphi_n$ on $C_n$ such that the restriction to $F_n$ is $\psi_n$ and let us show that this is also true for $n+1$. We set $A=C_n$, $I=F_{n+1}$, $B=C_{n+1}$, $\varphi=\varphi_n$ and $\psi=\psi_{n+1}$ on Lemma \ref{lemma_exel_laca}. By Proposition \ref{prop_intersection}, we have that $A\cap I=\overline{\mathrm{span}}\{s_{\mu}s_{\nu}^*:\mu,\nu\in E^{n},s(\mu)=s(\nu),|\mu|=|\nu|,0<|r^{-1}(s(\mu))|<\infty\}$. Let $s_{\mu}s_{\nu}^*\in A\cap I$. 
Since $0<|r^{-1}(s(\mu))|<\infty$ we have that
$$\psi(s_{\mu}s_{\nu}^*)=\psi_{n+1}(s_{\mu}s_{\nu}^*)=\sum_{e\in r^{-1}(s(\mu))}\psi_{n+1}(s_{\mu e}s_{\nu e}^*)=\sum_{e\in r^{-1}(s(\mu))}[\mu e=\nu e]c(\mu e)^{-\beta}\tau_{s(\mu e)}=$$
$$=\sum_{e\in r^{-1}(s(\mu))}[\mu=\nu]c(\mu)^{-\beta}c(e)^{-\beta}\tau_{s(e)}=[\mu=\nu]c(\mu)^{-\beta}\sum_{e\in r^{-1}(s(\mu))}c(e)^{-\beta}\tau_{s(e)}=$$
$$=[\mu=\nu]c(\mu)^{-\beta}\mathcal{F}_{c,\beta}(\tau)(p_{s(\mu)})=[\mu=\nu]c(\mu)^{-\beta}\tau(p_{s(\mu)})=\psi_n(s_{\mu}s_{\nu}^*)=\varphi_n(s_{\mu}s_{\nu}^*).$$
Again, using the approximate unit given by Lemma \ref{lemma_approxunit}, if $s_{\mu}s_{\nu}^*\in C_n$, then
$$\overline{\psi}(s_{\mu}s_{\nu}^*)=\overline{\psi_{n+1}}(s_{\mu}s_{\nu}^*)=\lim_{\lambda\to\infty}\psi_{n+1}(s_{\mu}s_{\nu}^*u_{\lambda})=\lim_{D\to r^{\leq n+1-|\nu|}(s(\mu))}\sum_{\zeta \in D}\psi_{n+1}(s_{\mu\zeta}s_{\nu\zeta}^*)=$$
$$=\lim_{D\to r^{\leq n+1-|\nu|}(s(\nu))}\sum_{\zeta \in D}[\mu\zeta=\nu\zeta]c(\nu \zeta)^{-\beta}\tau_{s(\nu\zeta)}=$$
$$=\lim_{D\to r^{\leq n+1-|\nu|}(s(\nu))}\sum_{\zeta \in D}[\mu=\nu]c(\nu)^{-\beta}c(\zeta)^{-\beta}\tau_{s(\zeta)}=$$
$$=[\mu=\nu]c(\nu)^{-\beta}\lim_{D\to r^{\leq n+1-|\nu|}(s(\nu))}\sum_{\zeta \in D}c(\zeta)^{-\beta}\tau_{s(\zeta)}\leq$$
$$\leq[\mu=\nu]c(\nu)^{-\beta}\tau_{s(\nu)}=\varphi_n(s_{\mu}s_{\nu}^*),$$
where the inequality is given by Lemma \ref{lemma_inequalityF}, which is a consequence of (K2).

By the description of the core $C^*(E)^{\gamma}$ as an inductive limit of the $C_n$, we can define a state $\omega$ as the inductive limit of $\varphi_n$. By construction, $\omega$ satisfies (\ref{eq:CoreCondition}) and, since each $\varphi_n$ is uniquely defined by (\ref{eq:CoreCondition}), so is $\omega$.

Finally, it is easily seen that the correspondence built preserves convex combinations by construction.

\end{proof}

\section{Ground states}\label{section_ground_states}

In this section, we let a function $c:E^1\to\mathbb{R}_+^*$ be given and define a one-parameter group of automorphisms $\sigma$ as in the last section.

The following definition of a ground state will be used \cite{MR1441540}.

\begin{definition}
We say that $\phi$ is a $\sigma$-\emph{ground state} if for all $a,b\in C^*(E)^a$, the entire analytic function $\zeta\mapsto \phi(a\sigma_{\zeta}(b))$ is uniformly bounded in the region $\{\zeta\in\mathbb{C}:\mathrm{Im}(\zeta)\geq 0\}$, where $C^*(E)^a$ is the set of analytic elements for $\sigma$.
\end{definition}

\begin{proposition}\label{prop:GroundState1}
If $\tau$ is a tracial state on $C_0(E^0)$ such that $\mathrm{supp}(\tau)\subseteq\{v\in E^0:v\ \mathrm{is\ singular}\}$ then there is a unique state $\phi$ on $C^*(E)$ such that 
\begin{itemize}
\item[(i)] $\phi(p_v)=\tau(p_v)$ for all $v\in E^0$;
\item[(ii)] $\phi(s_{\mu}s_{\nu}^*)=0$ if $|\mu|>0$ or $|\nu|>0$.
\end{itemize}
\end{proposition}

\begin{proof}
First, observe that a state $\phi$ satisfying (ii) is uniquely determined by its values on $C^*(E)^{\gamma}$ because (ii) implies that $\phi=\phi|_{C^*(E)^{\gamma}}\circ\Phi$, where $\Phi$ is the conditional expectation given by Proposition \ref{prop_conditional_expectation}.

Given $\tau$ as in the statement of the proposition, a state $\omega$ on $C^*(E)^{\gamma}$ can be built in the same way as in the proof of Theorem \ref{thm:KMSparte2}. For each $n$, use Lemma \ref{lemma_exel_laca} with $A=C_n$, $B=C_{n+1}$, $I=F_{n+1}$, $\psi_n\equiv 0$ and $\varphi_n$ is given by the previous step, where for the first step we have $\varphi_0=\tau$. For $\omega=\varinjlim \varphi_n$, we have that $\phi=\omega\circ\Phi$ satisfies (i) and (ii) and is unique by construction.
\end{proof}

\begin{proposition}\label{prop:GroundState2}
If $c$ is such that $c(e)>1$ for all $e\in E^1$, then a state $\phi$ on $C^*(E)$ is a $\sigma$-ground state for $\sigma$ if and only if $\phi(s_{\mu}s_{\nu}^*)=0$ whenever $|\mu|>0$ or $|\nu|>0$.
\end{proposition}

\begin{proof}
If $\phi$ is a ground state then for each pair $\mu,\nu\in E^*$ the function $\zeta\mapsto |\phi(s_{\mu}\sigma_{\zeta}(s_{\nu}^*))|$ is bounded on the upper half of the complex plane. If $\zeta = x+iy$ then
$$|\phi(s_{\mu}\sigma_{\zeta}(s_{\nu}^*))|=|\phi(s_{\mu}c(\nu)^{-i\zeta}s_{\nu}^*)|=|c(\nu)^{y-ix}\phi(s_{\mu}s_{\nu}^*)|=c(\nu)^y|\phi(s_{\mu}s_{\nu}^*)|.$$
If $|\nu|>0$, we have that $c(\nu)>1$ and so the only possibility for the above function to be bounded is if $\phi(s_{\mu}s_{\nu}^*)=0$. It is shown analogously that if $|\mu|>0$ then $\phi(s_{\mu}s_{\nu}^*)=0$.

For the converse, observe that if $|\mu|=|\nu|=0$ then $|\phi(s_{\mu}\sigma_{\zeta}(s_{\nu}^*))|=|\phi(s_{\mu}s_{\nu}^*|\leq 1$. It can be now readily verified that if $\phi(s_{\mu}s_{\nu}^*)=0$ whenever $|\mu|>0$ or $|\nu|>0$ then $\phi$ is a ground state.
\end{proof}

\begin{theorem}\label{thm:GroundStates}
If $c$ is such that $c(e)>1$ for all $e\in E^1$ then there is a bijective correspondence, given by restriction, between $\sigma$-ground states $\phi$ and tracial states $\tau$ on $C_0(E^0)$ such that $\mathrm{supp}(\tau)\subseteq\{v\in E^0:v\ \mathrm{is\ singular}\}$.
\end{theorem}

\begin{proof}
This is an immediate consequence of Propositions \ref{prop:GroundState1} and \ref{prop:GroundState2}. Just note that if $\phi$ is a $\sigma$-ground state and $v\in E^0$ is not singular then
$$\phi(p_v)=\phi\left(\sum_{e\in r^{-1}(v)} s_e s_e^* \right)=0.$$
\end{proof}

\section{Examples}\label{section_examples}

In this section we give two examples with infinite graphs and study the KMS states on the C*-algebras associated to these graphs.

\begin{example}[The Cuntz algebra $\mathcal{O}_{\infty}$]
Let $E^0=\{v\}$ be any unitary set and $E^1=\{e_n\}_{n\in\mathbb{N}}$ any countably infinite set with $r(e_n)=s(e_n)=v$ $\forall n\in\mathbb{N}$, then $C^*(E)\cong \mathcal{O}_{\infty}$.

If $c(e_n)=e$ (Euler's number) then we have the usual gauge action. In this case, $\mathcal{F}_{c,\beta}(\tau)(p_v)=\infty$ so that condition (K2) from Theorem \ref{thm:KMSparte2} is not satisfied and we have no KMS states for finite $\beta$. Since we have only one state on $C_0(E^0)$ and $v$ is a singular vertex, by Theorem \ref{thm:GroundStates} there exists a unique ground state.

Now if $c(e_n)=a_n$ where $a_n\in(1,\infty)$ is such that there is $\beta>0$ for which $\sum_{n=0}^\infty a_n^{-\beta}$ converges, then there exists $\beta_0>0$ such that $\sum_{n=0}^\infty a_n^{-\beta}=1$. Observing that $\mathcal{F}_{c,\beta}(\tau)(p_v)=\sum_{n=0}^\infty a_n^{-\beta}$ and using again the fact that there exists only one state on $C_0(E^0)$, we conclude from Theorems \ref{thm:KMSparte1} and \ref{thm:KMSparte2} that there is no KMS state for $\beta<\beta_0$, there exists a unique KMS state for each $\beta\geq \beta_0$ and, as with the gauge action, there is a unique ground state.
\end{example}

\begin{example}[A graph with infinitely many sources]
Let $E^0=\{v_n\}_{n\in\mathbb{N}}$ and $E^1=\{e_n\}_{n\in\mathbb{N}\setminus\{0\}}$ be countably infinite sets and define $r(e_n)=v_0$ and $s(e_n)=v_n$ for all $n\in\mathbb{N}\setminus\{0\}$.

Again, let $a_n\in(1,\infty)$, $n\in\mathbb{N}\setminus\{0\}$, be such that $\sum_{n=1}^\infty a_n^{-\beta}$ converges for some $\beta >0$. For $n\neq 0$ we have that $\mathcal{F}_{c,\beta}(\tau)(p_{v_n})=0$ and for $n=0$ we have $\mathcal{F}_{c,\beta}(\tau)(p_{v_0})=\sum_{n=1}^\infty a_n^{-\beta}\tau_{v_n}$. Condition (K1) of Theorem \ref{thm:KMSparte2} is trivially satisfied, and for condition (K2) we need $\sum_{n=1}^\infty a_n^{-\beta}\tau_{v_n}\leq \tau_{v_0}$.

If $\tau_{v_0}> 0$, since $0\leq\tau_{v_n}\leq 1$ for all $n$ there exists $\beta_0>0$ such that $\sum_{n=1}^\infty a_n^{-\beta_0}\tau_{v_n} = \tau_{v_0}$ so that (K2) is verified for all $\beta\geq \beta_0$ and so there are infinitely many KMS states. And for $\beta<\beta_0$ (K2) is not verified so that there are no KMS states.

For ground states, since all vertices are singular, we have no restriction on $\tau_{v_0}$; every state $\tau$ on $C_0(E^0)$ gives a ground state on $C^*(E)$.
\end{example}

\bibliographystyle{abbrv}
\bibliography{kms_states_ref}

\end{document}